\documentclass{amsart}
\usepackage{amssymb,latexsym}
\usepackage{amsmath}
\usepackage[dvipdfm]{graphicx}
\usepackage{enumerate}
\newenvironment{enumeratei}{\begin{enumerate}[\upshape (i)]}%
                            {\end{enumerate}}
\theoremstyle{plain}
 \newtheorem{theorem}{Theorem}[section]
 \newtheorem{lemma}[theorem]{Lemma}
\theoremstyle{definition}
  
  \newtheorem{notation}[theorem]{Notation}
  \newtheorem{remark}[theorem]{Remark}
  
\renewcommand \phi{\varphi}
\newcommand\set[1]{\{#1\}}
\newcommand \length [1] {\textup{length}(#1)}
\newcommand \skel [1] {S(#1)} %do not change since there are $S$ with not parameter
\newcommand \upskel [1] {S^{+}(#1)} %do not change since there are $S$ with not parameter
\newcommand \dskel [1] {S^{\scriptscriptstyle{\textup{wd}}}(#1)}
\newcommand \wt  {{\sf w}}
\newcommand \wpr  {{\sf w}'}
\newcommand \wof [2]  {{\wt}(#1,#2)}
\newcommand \wprof [2]  {{\wpr}(#1,#2)}
\newcommand \wxt  {{\sf w}^\ast}
\newcommand \wxof [2]  {{\wxt}(#1,#2)}
\newcommand \wxofprime [2]  {{\wxt}'(#1,#2)}
\newcommand \Bigwxof [2]  {{\wxt}\Bigl(#1,#2\Bigr)}
\newcommand \Bigwxofprime [2]  {{\wxt}'\Bigl(#1,#2\Bigr)}
\newcommand \At [1] {{\sf{ At}} (#1)}
\newcommand \Ji [1] {{\sf{ Ji}} (#1)}
\newcommand \Jki [2] {{\sf{ Ji}}_{#1} (#2)}
\newcommand \Jtop [1]{{ \sf{Ji}^{+}(#1)}}
\newcommand \domin [1] {\textup{domin}(#1)}
\newcommand \gd [1] {d_{#1}}
\newcommand \gs [1] {s_{#1}}
\newcommand \gee [1] {e_{#1}}
\newcommand \nopdstar[1] {d^{\sharp}_{#1}}
\newcommand \dstar [2] {\nopdstar{#1}\kern-1pt(#2)}
\newcommand \nopsstar[1] {s^{\sharp}_{#1}}
\newcommand \sstar [2] {\nopsstar{#1}\kern-1pt(#2)}
\newcommand \nopestar[1] {e^{\sharp}_{#1}}
\newcommand \estar [2] {\nopestar{#1}\kern-1pt(#2)}
\newcommand \pinf [1] {P^{\infty}\kern-1pt(#1)}
\newcommand \ekval [1] {\mathrel{\mathord=^{#1} }}
\newcommand \clss [1] {{\mathcal{ #1}}}

\begin{document}

\title%[Distributive lattices by weighted double skeletons]
{Distributive lattices determined by weighted double skeletons}

\author[G.\ Cz\'edli]{G\'abor Cz\'edli}
\email{czedli@math.u-szeged.hu}
\urladdr{http://www.math.u-szeged.hu/~czedli/}
\address{University of Szeged, HUNGARY 6720}
%% Second author (Note: The order of the items here is important!)
\author[J. Grygiel]{Joanna Grygiel}
\email{j.grygiel@ajd.czest.pl}
\urladdr{http://www.j.grygiel.eu/}
\address{Institute of Mathematics and Computer Science, 
Jan D{\l}ugosz University of Cz\c{e}stochowa, Poland%
%Jan D\l ugosz University of Cz\c estochowa, Poland%
}
%% Third author 
\author[K. Grygiel]{Katarzyna  Grygiel}
\email{grygiel@tcs.uj.edu.pl}   
\urladdr{http://tcs.uj.edu.pl/Grygiel}
\address{
Theoretical Computer Science Department, 
Faculty of Mathematics and Computer Science, 
Jagiellonian University,
ul. Prof. {\L}ojasiewicza 6, 30-348 Krak\'ow, Poland}


\date{October 27, 2011; revised August 21, 2012; September 5, 2012; October 12, 2012.}

\begin{abstract} Related to his $S$-glued sum construction, the skeleton $\skel L$ of a finite lattice $L$ was introduced by C.\ Herrmann in 1973. Our theorem asserts that if $D$
 is a finite  distributive lattice  and its second skeleton, $\skel{\skel D}$, is the trivial lattice, then  $D$ is characterized 
by its weighted double skeleton, introduced by the second author in 2006. 
The assumption on the second skeleton is essential.
%If the size $|D|$ of $D$ is much larger than the number of maximal boolean intervals of $D$, 
%then our characterization offers the most concise known description of $D$.
\end{abstract}

%% Thanks
\thanks{This research of the first author was supported by the NFSR of Hungary (OTKA), grant numbers   K77432 and K83219. The second author was supported by the NSC of Poland, grant number 2011/01/B/HS1/00944. The third author was supported by the NSC of Poland, grant number 2011/01/B/HS1/00944, and by the Polish Ministry of Science and Higher Education, grant number NN206 376137.}

\maketitle

\section{Introduction} Let $L$ be a finite modular lattice. Then, according to Herrmann~\cite{herrmann}, $L$ is the union of its maximal complemented (equivalently, atomistic) intervals, which are glued together 
along a lattice $S=\skel L$, the \emph{skeleton} of $L$. 
His construction of $S(L)$ makes sense even without  modularity, so we drop this assumption until otherwise stated.
It appeared somewhat later that $\skel L$ is a factor lattice of $L$ by a tolerance relation in the sense of the first author \cite{czedli}. Define $S^0(L):=L$ and $S^{i+1}(L):= S\bigl(S^i(L)\bigr)$. Then there is a smallest $n$ such that $|S^n(L)|=1$, which we will call the \emph{Herrmann rank of $L$}. We say that $L$ is 
\emph{$H^n$-irreducible} if its Herrmann rank  is at most $n$. Equivalently, $L$ is $H^n$-irreducible if{f} $|S^n(L)|=1$. 
$H^1$-irreducibility was previously called  $H$-irreducibility  by the second author in the monograph \cite{grygiel_mono} and in many of her papers, including \cite{grygiel_mono} and \cite{grygiel_double}. 

The skeleton of $L$ does not tell too much on $L$. Indeed, the second author \cite[Corollary 3.2.6]{grygiel_mono} proved that each finite lattice $S$ is the skeleton of infinitely many pairwise non-isomorphic finite distributive lattices. The \emph{weighted double skeleton} $\dskel L$ of $L$, introduced by the second author in 
\cite{grygiel_double} and to be  defined in  the present paper soon, carries much more information on the initial lattice. 

Let $\clss K$ be a class of finite distributive lattices, and let $L\in \clss K$. If for any $L'\in \clss K$ such that  $\dskel {L'}$ is isomorphic to $\dskel{L}$
the lattice $L'$ is isomorphic to $L$, then  we say that $L$ is \emph{determined} by its weighted double skeleton \emph{in the class} $\clss K$.

As usual,  the partially ordered set (in short, the poset, in other words, the order) of all non-zero join-irreducible elements of $L$ is denoted by   $\Ji D=(\Ji D,\leq)$. 
The sets $\set{1,2,3,\dots}$ and $\set{0,1,2,3,\dots}$ are denoted by $\mathbb N$ and $\mathbb N_0$, respectively. 
The \emph{length} of a finite poset $Q$ is
$\length Q:=\max\{n\in\mathbb N_0: Q$ has an $n+1$-element chain$\}$. A \emph{nontrivial lattice} is a lattice that has at least two elements.
Postponing the rest of definitions to the next section, our main result reads as follows. 

\begin{theorem}\label{thmmain} 
Let $L$ be a  finite nontrivial  lattice. 
\begin{enumeratei}
\item\label{thmmaina} If $L$ is modular and $H^2$-irreducible, then $\length{\Ji L}\leq 1$.
\item\label{thmmainb} If $L$ is distributive and $H^2$-irreducible, then  $L$ is  determined  by its weigh\-ted double skeleton in the class of finite distributive lattices. 
\item\label{thmmainc} If $L$ is distributive and  $\length{\Ji L}\leq 1$, then $L$ is  determined  by its weighted double skeleton in the class of finite distributive lattices $D$ satisfying the inequality $\length{\Ji D}\leq 1$. 
\end{enumeratei}
\end{theorem}

Notice that $\dskel L$ determines some properties of a modular $L$ even if $H^2$-irreducibility is not assumed.  Namely, $\dskel L$ clearly determines $\length L$, and see Lemma~\ref{nojiralpha} for further properties in the distributive case. 
However, we will soon prove
the following remark, which  indicates that
Theorem~\ref{thmmain} is optimal in some sense.

\begin{remark}\label{notforrankthree}
There exist $H^3$-irreducible  finite distributive lattices $L_1$ and $L_2$ such that 
$\dskel{L_1}$ is isomorphic to $\dskel{L_2}$ but $L_1$ is not isomorphic to $L_2$.
Also, there is a finite distributive lattice $L$ such that $\length{\Ji L}\leq 1$ but $L$ is not $H^2$-irreducible.
\end{remark}

A well-known economic way of describing a finite distributive lattice $D$ by a little amount of  data is to consider  $\Ji D=(\Ji D,\leq)$.  The next remark outlines a more economic way for certain distributive lattices.

\begin{remark}\label{RmRkxg} Let $D$ be a finite distributive lattice with  $\length{\Ji D}\leq 1$. Assume that  $D$ is the union of few maximal boolean intervals but $|D|$ is large. Then $\dskel D$ constitutes an economic description of $D$. 
\end{remark}

\section{Basic concepts and  statements}
For the basic concepts of Lattice Theory the reader is referred to Gr\"atzer~\cite{ggLT}. By a \emph{tolerance} of a lattice $L$ we mean a reflexive, symmetric, compatible relation of $L$. Equivalently, a tolerance of $L$ is the image of a congruence 
by a surjective lattice homomorphism onto $L$, see the first author and Gr\"atzer~\cite{czggg}. 
Let $R$ be a tolerance of $L$. If $X\subseteq L$ is a maximal subset with respect to the property $X\times X\subseteq R$, then $X$ is called a \emph{block} of $R$.  Blocks are convex sublattices by  Bandelt~\cite{bandelt} and Chajda~\cite{chajdabook}. Let $\alpha$ and $\beta$ be blocks of $R$. As it follows immediately from Zorn's Lemma, there are blocks $\gamma$ and $\delta$ of $R$ such that 
\begin{equation}\label{joinmeet}
\begin{aligned}
\set{x\vee y:x\in\alpha,\,\, y\in \beta}\subseteq \gamma=: \alpha\vee\beta,\cr
%\quad\text{and}\quad  
\set{x\wedge y:x\in\alpha,\,\, y\in \beta}\subseteq \delta=:\alpha\wedge\beta\text.
\end{aligned}
\end{equation}
The first author \cite{czedli} proved that $\gamma$ and $\delta$ are uniquely determined, and the set $L/R$ of all blocks of $R$ with the join and meet defined by \eqref{joinmeet} is a lattice. 
This lattice, also denoted by $L/R$, is called the \emph{factor lattice} (or quotient lattice) of $L$ modulo $R$. 
Notice that there is an alternative way, which does not rely on the axiom of choice (and, therefore, on Zorn's Lemma), to define $L/R$ in an order-theoretic way and to prove that it is a lattice, see  Gr\"atzer and Wenzel~\cite{ggwenzel}.

In the rest of the paper, all lattices will be assumed to be finite. Then the blocks of a tolerance $R$ are intervals. So if $\alpha$ is a block of $R$, then $\alpha$ equals the interval $[0_\alpha,1_\alpha]$ of $L$. It was proved in \cite{czedli} and \cite{czkluk} that, for all $\alpha,\beta\in L/R$,
\begin{equation}\label{blockscompute}
\begin{aligned}
&0_\alpha\vee 0_\beta = 0_{\alpha\vee\beta},\quad  
1_\alpha\vee 1_\beta \leq  1_{\alpha\vee\beta},\cr  
&1_\alpha\wedge 1_\beta = 1_{\alpha\wedge\beta},\quad 
0_\alpha\wedge 0_\beta \geq 0_{\alpha\wedge\beta},\cr
&\alpha\leq\beta\text{ (in }L/R\text{)}\iff   0_\alpha\leq 0_\beta \iff 1_\alpha\leq 1_\beta
\text. 
\end{aligned}
\end{equation}
The most important particular case of $L/R$, under the name skeleton, was discovered by Herrmann~\cite{herrmann} much earlier; we survey it partly and only for the finite case. A tolerance $R$ of (a finite lattice) $L$ is called a \emph{glued tolerance}, see Reuter~\cite{reuter}, if 
its transitive closure $R^\ast$ is the total relation $L^2$. The (unique) smallest glued tolerance of $L$ is  called the \emph{skeleton tolerance} of $L$, and it is denoted by $\Theta(L)$. There are two easy ways to see that $\Theta(L)$ exists. Firstly, we know from the second author~\cite{grygiel_mono}, and it  is routine to check, that for any tolerance $R$ of a finite lattice $L$,
\begin{equation}\label{gluedtol}
R \text{ is a glued tolerance} \iff 
(x,y)\in R \text{ for all }x\prec y\in L\text. 
\end{equation}
This clearly implies that the intersection of all glued tolerances of $L$ is a glued tolerance again, whence it is the skeleton tolerance of $L$. Secondly, we know from \cite{cz_h_r} that the transitive closure of lattice tolerances commutes with their (finitary) intersections, which also implies the existence of $\Theta(L)$. 

\begin{figure}
\centerline
{\includegraphics[scale=1.0]{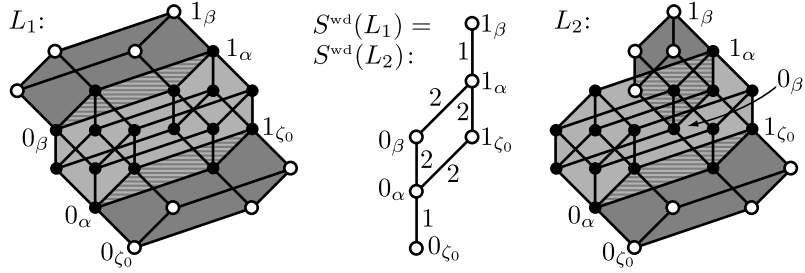}}
\caption{Non-isomorphic lattices with isomorphic double skeletons}\label{figone}
\end{figure}

The factor lattice $L/\Theta(L)$ is called the \emph{skeleton} $\skel L$ of $L$.  We claim that 
$\length{\skel L}<\length L$. Indeed, assume that 
$k=\length{\skel L}$ and $\alpha_0\prec \alpha_1\prec\cdots\prec \alpha_k$ is a maximal chain in $\skel L=L/\Theta(L)$. It follows from 
\eqref{blockscompute} and \eqref{gluedtol}  that $0_{\alpha_0}<0_{\alpha_1}<\cdots< 0_{\alpha_k}<1_{\alpha_k}$, showing that $\length{\skel L}=k<k+1\leq \length L$. The inequality  $\length{\skel L}<\length L$ shows that each finite lattice has a Herrmann rank.

It is clear from \eqref{blockscompute} that both $\set{0_\alpha: \alpha\in \skel L}$ and $\set{1_\alpha: \alpha\in \skel L}$, as sub-posets of $L$, are order isomorphic to $\skel L$. Their union carries a lot of information on $L$ provided we equip it with an appropriate structure. 
Following the second author \cite{grygiel_double}, a structure 
$(P,\leq, K, \eta_0,\eta_1, \wt)$ will be  called an \emph{abstract weighted double skeleton} if $(P,\leq)$ is a finite poset, $K$ is a lattice, $\eta_0\colon K\to P$ 
is a join-preserving (and, therefore, order-preserving) embedding, 
$\eta_1\colon K\to P$ is a meet-preserving order-embedding,  
$P=\eta_0(K)\cup \eta_1(K)$,
$\eta_0(x)\leq\eta_1(x)$ holds for all $x\in K$,
and $\wt$ is a mapping of the covering relation $\set{(a,b)\in P^2: a\prec b}$ into 
$\mathbb N$. The underlying set of $(P,\leq, K, \eta_0,\eta_1, \wt)$  is $P$, and we often denote the structure $(P,\leq, K, \eta_0,\eta_1, \wt)$ simply by $P$.

Let $(P',\leq',K', \eta_0',\eta_1', \wpr )$ be another abstract  weighted double skeleton, and let $(\psi,\kappa)$ be a pair of bijective mappings. We say that 
\[ (\psi,\kappa)\colon  (P,\leq, K, \eta_0,\eta_1, \wt)
\to (P',\leq',K', \eta_0',\eta_1', \wpr )
\]
is an \emph{isomorphism} if
$\psi\colon(P,\leq)\to (P',\leq')$ is an order isomorphism, $\kappa\colon K\to K'$ is a lattice isomorphism,  
$\psi\bigl(\eta_i(x)\bigr)=\eta_i'\bigl(\kappa(x)\bigr)$ for all $x\in K$ and $i\in\set{0,1}$, and $\wof xy=\wprof{\psi(x)}{\psi(x)}$ for all $x,y\in P$ such that $x\prec y$. If there is such a $(\psi,\kappa)$, then the two abstract  weighted double skeletons are called \emph{isomorphic}.
By the (concrete) \emph{weighted double skeleton} of $L$ we mean  the structure
\begin{equation}\label{dkselL}
\dskel L:=\bigl(\set{0_\alpha: \alpha\in \skel L} \cup\set{1_\alpha: \alpha\in \skel L},\leq , \skel L,\eta_0,\eta_1,\wt\bigr)
\end{equation}
where $\leq$ is the ordering inherited (restricted) from $L$, $\eta_0(\alpha):=0_\alpha$ and $\eta_1(\alpha):=1_\alpha$ for all $\alpha\in \skel L$, and 
$\wof x y:=\length{[x,y]_L}$ for any $x\prec_{\dskel L}y$. 

For example, consider $L_j$ given in Figure~\ref{figone}    for $j\in\set{1,2}$. Then $\skel {L_j}$ is the three-element chain $\set{\zeta_0\prec\alpha\prec \beta}$, and $\dskel{L_j}$ is depicted in the middle of the figure. For $x\prec y$ in $\dskel{L_j}$, the edge $x\prec y$ of the diagram is labeled by $\wof x y$.
Since $L_1\not\cong L_2$ and 
$S^n(L_j)$ is the $(4-n)$-element chain for $n\in\set{1,2,3}$ and $j\in\set{1,2}$, 
 Figure~\ref{figone} together with the self-explanatory Figure~\ref{figthree} proves Remark~\ref{notforrankthree}.

We have defined all the concepts Theorem~\ref{thmmain} is based on. The rest of the paper is devoted to proofs, including some auxiliary statements.

\begin{figure}
\centerline
{\includegraphics[scale=1.0]{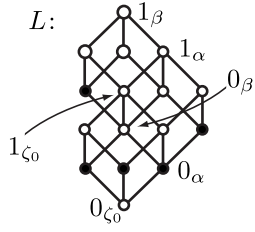}}
\caption{$\length{\Ji L}\leq 1$ 
does not imply $H^2$-irreducibility}  \label{figthree}
\end{figure}

\section{Proofs and auxiliary statements}
\subsection{The number of join-irreducible elements in a block}
Unless otherwise stated, by a \emph{block} of a lattice $L$ we mean a block of its skeleton tolerance $\Theta(L)$, that is, a member of the skeleton $\skel L$.
Throughout this subsection, $L$ denotes a finite modular lattice. 
We are going to extend the weight function $\wt$ of $\dskel L$, see \eqref{dkselL},  to a function $\wxt\colon \dskel L\times   \dskel L\to \mathbb N_0$. If
$x,y\in \dskel L$ and $x\not< y$, then we let $\wxof xy:=0$. If $x,y\in\dskel L$ and  $x < y$, then  take a maximal chain $z_0=x\prec_{\dskel L} z_1 \prec_{\dskel L}  \cdots \prec_{\dskel L} z_t=y$ in $\dskel L$, and define 
\[\wxof xy := \sum_{i=0}^{t-1}\wof {z_i}{z_{i+1}}\text.
\]
The lattice theoretical Jordan-H\"older theorem applies on $L$ and we conclude that $\wof x y:=\length{[x,y]_L}$ for any $x\leq_{\dskel L}y$. This  guarantees  
that $\wxof xy$ does not depend on the maximal chain chosen. 
Given a poset $(Q,\leq)$, the \emph{M\"{o}bius function} $\mu_Q \colon Q^2 \to {\mathbb Z}$ is defined recursively as follows:
\[
\mu_Q(x,y) = \begin{cases} 1,&\text{if $x=y$;}\\
 \displaystyle{-\sum_{x \leq z < y} \mu_Q(x,z)},&\text{if $x<y$;}\\
0,&\text{if $x\not\leq y$.}
\end{cases}
\]

If $L$ is a finite modular lattice, then every $\alpha\in \skel L$ is an atomistic lattice by Herrmann~\cite{herrmann}. In this case,  $\Ji{\alpha}$ stands for the set of join-irreducible elements of $\alpha$ that are distinct from $0_\alpha$. The set of atoms of $\alpha$ is denoted by $\At\alpha$. We let
\begin{equation}\label{jalpadeF}
J_\alpha:=\Ji L \cap \Ji\alpha=\Ji L\cap \At \alpha=\Ji L\cap (\alpha\setminus\set{0_\alpha})\text.
\end{equation}
Since $\dskel L$ determines $\skel L$, the next lemma, based on the notation above, implies that $|J_\alpha|$ is determined by $\dskel L$, provided $L$ is distributive.  
\begin{lemma}\label{nojiralpha}
Let $L$ be  a finite distributive lattice. Then for each $\alpha\in \skel L$, 
\begin{equation}\label{fnojiralpa}
|J_\alpha| = \sum_{\beta \leq \alpha, \,\, \beta \in \skel L} \mu_{\skel L}(\beta , \alpha) \cdot \wxof {0_\alpha} {1_\beta}\text.
\end{equation}
\end{lemma}
  
\begin{proof} For $k\in \mathbb N$ and a poset $K$, let $\Jki kK$ denote the set of elements of $K$ with exactly $k$ lower covers. Notice that $\Ji L=\Jki 1L$. Reuter~\cite[Corollary 3]{reuter} asserts, even when  $L$ is only modular, that for each $\alpha\in \skel L$ and $k\in\mathbb N$, 
\begin{equation}\label{reuterfnoj}
|\Jki kL\cap\Jki k{\alpha}| = \sum_{\beta \leq \alpha, \,\, \beta \in \skel L} \mu_{\skel L}(\beta , \alpha) \cdot |\Jki k{\alpha\cap\beta}|\text.
\end{equation}
For $k=1$, the lefthand side of \eqref{fnojiralpa} equals  that  of \eqref{reuterfnoj}.
Hence it suffices to show that $\wxof {0_\alpha} {1_\beta}=|\Jki 1{\alpha\cap\beta}|$ holds  for $\beta\leq\alpha$. This is obvious if $\alpha\cap\beta=\emptyset$ since then $0_\alpha\not\leq 1_\beta$ follows from \eqref{blockscompute}. Hence we assume that $\alpha\cap\beta\neq \emptyset$. Then, again by \eqref{blockscompute}, $\alpha\cap\beta= [0_\alpha,1_\beta]$  is a principal ideal of $\alpha$, whence $\alpha\cap\beta$ is a boolean sublattice of the boolean interval $\alpha$. Thus,
\[|\Jki 1{\alpha\cap\beta}|=
|\Ji {[0_\alpha,1_\beta]}|=
|\At {[0_\alpha,1_\beta]}|=\length{[0_\alpha,1_\beta]} = \wxof {0_\alpha}{1_\beta}\text.\qedhere
\]
\end{proof}

\subsection{More about blocks}

Although the following lemma requires a proof in the present setting, it is a part of the original definition of $\skel L$ given by  Herrmann~\cite{herrmann}. For the reader's convenience, we present an easy proof.

\begin{lemma}[{Wille~\cite[Proposition 9]{wille}}]\label{precbloists} 
Let $L$ be a finite lattice. 
If $\alpha\prec_{\skel L}\beta$, then $\alpha\cap\beta\neq\emptyset$.
\end{lemma}

\begin{proof} Assume that $\alpha\prec_{\skel L}\beta$, and let $a:=1_\alpha$.  Since  $1_\alpha< 1_\beta$ by \eqref{blockscompute}, we can take an element $b\in L$ such that $a=1_\alpha\prec b\leq 1_\beta$. 
Since $L=\bigcup_{{\nu}\in \skel L}{\nu}$, there is a $\gamma_1\in \skel L$ containing $b$. By  \eqref{joinmeet}, $\gamma_2:=\alpha\vee \gamma_1$ contains $b=a\vee b$, and $\gamma_3:=\beta\wedge \gamma_2$ also contains $b=1_\beta\wedge b$. Clearly, $\alpha\leq \gamma_3\leq \beta$. Taking $b\in \gamma_3\setminus \alpha$ and $\alpha \prec_{\skel L}\beta$ into account, we conclude that $b\in\gamma_3=\beta$. 

Next, $\set{a,b}^2\subseteq\Theta(L)$ since $a\prec b$. Hence there is a block $\gamma\in\skel L$ such that $\set{a,b}\subseteq \gamma$. 
Using  \eqref{joinmeet} repeatedly, we obtain that $a=a\vee a\in \alpha\vee \gamma$, $b=a\vee b\in \alpha\vee \gamma$, $a=a\wedge b\in (\alpha\vee \gamma)\wedge \beta$ and $b=b\wedge b\in (\alpha\vee \gamma)\wedge \beta$. 
That is, $\set{a,b}\subseteq (\alpha\vee \gamma)\wedge \beta$. On the other hand, $\alpha\leq (\alpha\vee \gamma)\wedge \beta \leq\beta$ together with $\alpha\prec_{\skel L}\beta$ yields that $\alpha= (\alpha\vee \gamma)\wedge \beta$ or  $\beta= (\alpha\vee \gamma)\wedge \beta$. Hence $\set{a,b}\subseteq \alpha$ or 
$\set{a,b}\subseteq \beta$, and we conclude that $\alpha\cap\beta\neq\emptyset$. 
\end{proof}

For the reader's convenience, we also prove the following lemma. Due to the forthcoming formula \eqref{tRmdoptrho},
the present approach is slightly simpler than the original one of the second author \cite{grygiel_mono}.

\begin{lemma}[{\cite[Theorem 2.2.10]{grygiel_mono}}]\label{ifTcLabT}
Let $L$ be a finite lattice. Assume that $\alpha,\beta\in \skel L$ such that 
$(\alpha,\beta)\in\Theta(\skel L)$. Then $\alpha\cap\beta\neq\emptyset$.
\end{lemma}

\begin{proof} A straightforward induction based on \eqref{joinmeet}  shows that, for any lattice term $p(x_1,\ldots,x_n)$ and for any $\nu_1,\ldots,\nu_n \in S:=\skel L$,
\begin{equation}\label{tRmdoptrho}
\set{p(x_1,\ldots,x_n): x_1\in\nu_1,\ldots, x_n\in\nu_n}\subseteq p(\nu_1,\ldots,\nu_n)\text.
\end{equation}

Next, let $S=\set{\sigma_1,\ldots,\sigma_{t}}$, and let $\kappa_1\prec_S \lambda_1$, $\ldots$, $\kappa_k\prec_S \lambda_k$ be a list of all covering pairs of $S$. Since the skeleton tolerance $\Theta(S)$ of $S$ is generated by $\set{(\kappa_1,\lambda_1),\ldots, (\kappa_k,\lambda_k)}$, it coincides with the subalgebra of $S^2$ generated by
\[\set{(\kappa_1,\lambda_1),\ldots, (\kappa_k,\lambda_k), (\lambda_1,\kappa_1),\ldots (\lambda_k,\kappa_k),(\sigma_1,\sigma_1),\ldots,(\sigma_{t},\sigma_{t})}\text.
\]
Hence there exists a $(2k+t)$-ary lattice term $p$ such that 
{\allowdisplaybreaks
\begin{align*}
(\alpha,\beta)&=p\bigl( (\kappa_1,\lambda_1), \ldots,(\kappa_k,\lambda_k),(\lambda_1,\kappa_1), \ldots, (\lambda_k,\kappa_k), (\sigma_1,\sigma_1), \ldots,(\sigma_{t}, \sigma_{t})\bigr)\cr
&=\bigl(p(\kappa_1,\ldots,\kappa_k,\lambda_1,\ldots,\lambda_k,\sigma_1,\ldots,\sigma_{t}),\cr
& \kern 18pt p(\lambda_1,\ldots,\lambda_k,\kappa_1,\ldots,\kappa_k,\sigma_1,\ldots,\sigma_{t})\bigr)\text.
\end{align*}%
}%
It follows from Lemma~\ref{precbloists} that there are $x_1,\ldots,x_k,y_1,\ldots,y_t\in L$ such that $x_i\in \kappa_i\cap\lambda_i =\lambda_i\cap \kappa_i$ for $i=1,\ldots,k$ and $y_j\in \sigma_j$ for $j=1,\ldots,t$. Hence the above expression for $(\alpha,\beta)$ together with \eqref{tRmdoptrho}
yields that 
\[p(x_1,\ldots,x_k, x_1,\ldots,x_k, y_1,\ldots,y_t)\in\alpha\cap\beta\text. \qedhere
\]
\end{proof}

\subsection{More about join-irreducible elements in blocks}
\begin{notation}\label{notdomlT}
Let $\Jtop L:=\Ji L\setminus \At L$. Let $\zeta_0=[0,z_0]$ be the least element of $\skel L$, and let 
$\upskel L:=\skel L\setminus\set{\zeta_0}$.
For $x\in L$, $\domin x:=\set{y\in\At L: y\leq x}$ is called the set of atoms \emph{dominated} by $x$. Similarly, for $\alpha\in \upskel L$, $\domin\alpha :=\set{y\in\At L: y\leq 0_\alpha}$ is the set of atoms dominated by $\alpha$.
\end{notation}

The next lemma is easy. Having no reference at hand, we will give  a proof.

\begin{lemma}\label{joiNatmszn} Let $L$ be a finite modular lattice. Then 
$[0,\bigvee_{x\in\At L}x]=\zeta_0$. Furthermore, if $\alpha\in \skel L$ such that $J_\alpha\cap \At L\neq \emptyset$, then $\alpha=\zeta_0$. 
\end{lemma}

\begin{proof} Let $z_0:=\bigvee_{x\in\At L}x$. Since $(0,x)\in \Theta(L)$ for all $x\in\At L$, we obtain that $(0,z_0)\in \Theta(L)$. Hence we can extend $\set{0,z_0}$ to a block $\alpha=[0,y]$ of $\Theta(L)$. Obviously, $z_0\leq y$. We know from Herrmann~\cite{herrmann} that $\alpha$ is an atomistic lattice. Hence 
$y=\bigvee_{x\in \At{\alpha}}x\leq \bigvee_{x\in \At L}x=z_0,
$ 
and we conclude that $[0,z_0]=[0,y]=\alpha\in  \skel L$. It is the smallest element of $\skel L$, that is $\zeta_0$,  by \eqref{blockscompute}.
Finally, if an atom $a\in \At L$ belongs to $J_\alpha$, then $0_\alpha\prec a$ implies $0_\alpha=0=0_{\zeta_0}$, whence $\alpha=\zeta_0$ by \eqref{blockscompute}.
\end{proof}

\begin{lemma}\label{kHhDwlP} Let $L$ be a finite modular lattice. Then 
\begin{enumeratei}
\item\label{kHhDwlPa} $\Ji L=\bigcup\set{J_\alpha: \alpha\in\skel L}$;
\item\label{kHhDwlPb} for all $\alpha,\beta\in \skel L$, if $\alpha\neq\beta$, then $J_\alpha\cap J_\beta = \emptyset$;
\item\label{kHhDwlPc} $\Jtop L=\bigcup\set{J_\alpha: \alpha\in\upskel L}$ and  $\At L=J_{\zeta_0}=\At{\zeta_0}$.
\end{enumeratei}
\end{lemma}

\begin{proof} Assume that $a\in \Ji L$, and let $a^-$ stand for its unique lower cover.
Then $\set{a^-,a}\subseteq \Theta(L)$ since  $\Theta(L)$  is a glued tolerance.
We can extend $\set{a^-,a}$ to a block $\alpha\in \skel L$. Then $a\in \Ji L \cap   (\alpha\setminus\set{0_{\alpha}})=J_\alpha$. This proves that $\Ji L\subseteq\bigcup\set{J_\alpha: \alpha\in\skel L}$. The reverse inclusion in part~\eqref{kHhDwlPa} is trivial.

Assume that $x\in J_\alpha\cap J_\beta$. Then, by \eqref{jalpadeF}, $x\succ 0_\alpha$ and $x\succ 0_\beta$. Hence $x\in\Ji L$ implies that $ 0_\alpha=0_\beta$. This together with \eqref{blockscompute} yields $\alpha=\beta$, which proves part~\eqref{kHhDwlPb}. Finally, parts~\eqref{kHhDwlPa} and \eqref{kHhDwlPb} together with Lemma~\ref{joiNatmszn} imply part~\eqref{kHhDwlPc}.
\end{proof}

\begin{lemma}\label{irrPar} Let $\alpha$, $\beta$ and $\zeta$ be distinct blocks of a finite modular lattice $L$ such that $\zeta<\alpha$, $\zeta<\beta$ and $\set{\alpha,\beta,\zeta}^2\subseteq \Theta\bigl(\skel L\bigr)$. Then $a\parallel b$  holds for all $a\in J_\alpha$ and $b\in J_\beta$.
\end{lemma}

\begin{proof}
Assume that $a$ is comparable with $b$. 
It follows from Lemma~\ref{kHhDwlP}\eqref{kHhDwlPb} that $a\neq b$. Hence we can assume that $a<b$.  
We infer $\alpha\cap\zeta\neq\emptyset$ and  
$\beta\cap\zeta\neq\emptyset$ by Lemma~\ref{ifTcLabT}. 
This together with  $\zeta<\alpha$ and $\zeta<\beta$ yields that 
$0_\zeta<0_\alpha\leq 1_\zeta$ and  
$0_\zeta<0_\beta\leq 1_\zeta$. Since  $a\in J_\alpha\subseteq \At\alpha$ by \eqref{jalpadeF}, we have that $0_\alpha\prec a$. Similarly,  $0_\beta\prec b$. 
Since $b\in J_\beta\subseteq\Ji L$, $0_\beta$ is the only lower cover of $b$. This together with $a<b$ implies that $a\leq 0_\beta$.
Therefore, 
$0_\zeta<0_\alpha \prec  a \leq 0_\beta\leq 1_\zeta$. This means that $a\in\zeta$ but, since $\Ji\zeta=\At \zeta$, $a$ is join-reducible in $\zeta$, whence it is also join-reducible in $L$. This contradicts $a\in J_\alpha\subseteq \Ji L$. 
\end{proof}

\begin{proof}[Proof of Theorem~\ref{thmmain}\eqref{thmmaina}]
Let us assume for a contradiction that 
there exist 
$a_1,a_2,a_3\in\Ji L$ such that $a_1<a_2<a_3$. 
By Lemma~\ref{kHhDwlP}\eqref{kHhDwlPa}, we can choose 
$\alpha_1,\alpha_2,\alpha_3\in \skel L$ such that  $a_i\in J_{\alpha_i}$, for $i\in\set{1,2,3}$. Since  $J_{\alpha_i} \subseteq \At{\alpha_i}$ is an antichain for $i=1,2,3$,  the blocks $\alpha_1,\alpha_2,\alpha_3$ are pairwise distinct. Hence at least two of them, say $\alpha_j$ and $\alpha_k$,  are distinct from the smallest element $\zeta_0$ of $\skel L$. Since 
 $S^2(L)= \skel L/\Theta\bigl(\skel L \bigr)$ is the singleton lattice, $\Theta(\skel L)$ is the full relation on $\skel L$.
Therefore, applying  Lemma~\ref{irrPar} to $\alpha_j$, $\alpha_k$ and $\zeta_0$, we obtain that $a_j\parallel a_k$, a contradiction.
\end{proof}

Clearly,  $|\Jtop L|\neq \emptyset$ if{f} $\length {\Ji L}\geq 1$. Hence the next lemma would  (vacuously) also hold if $\length{\Ji L}$ was $0$. Notation~\ref{notdomlT} will be in effect.

\begin{lemma}\label{iSfRg} Let $L$ be a finite distributive lattice such that $\length{\Ji L}=1$. Let $\alpha,\alpha_1,\ldots,\alpha_n\in\upskel L$ such that none of $J_\alpha, J_{\alpha_1}, \ldots, J_{\alpha_n}$ is empty. Then
\begin{enumeratei}
\item\label{iSfRga} $\domin{\alpha}\neq\emptyset$;
\item\label{iSfRgb} $\alpha_1\leq \alpha_2\,$ if{f} $\,\domin{\alpha_1}\subseteq\domin{\alpha_2}$;
\item\label{iSfRgc}  $|\domin{\alpha_1}\cup\ldots\cup\domin{\alpha_n}|=\wxof 0 {0_{\alpha_1\vee \cdots\vee \alpha_n}}$. 
\end{enumeratei}
\end{lemma}

\begin{proof}Since $0\neq 0_{\alpha}$, part \eqref{iSfRga}  is trivial. 
If $\alpha_1\leq \alpha_2$, then 
$0_{\alpha_1}\leq 0_{\alpha_2}$
by \eqref{blockscompute}, whence 
$\domin{\alpha_1}\subseteq\domin{\alpha_2}$. 
To prove the reverse implication of \eqref{iSfRgb}, assume that $\alpha_1\not\leq\alpha_2$. Then \eqref{blockscompute} implies $0_{\alpha_1}\not\leq 0_{\alpha_2}$, which yields an $x\in\Ji L$ such that $x\leq 0_{\alpha_1}$ but $x\not\leq 0_{\alpha_2}$. 
Since $J_{\alpha_1}\neq\emptyset$ by the assumption,
there is a $y\in\Ji L$ such that $x\leq 0_{\alpha_1}<y$. This together with $\length{\Ji L}=1$ shows that $x\in \At L$. Hence $x\in\domin{\alpha_1}\setminus\domin{\alpha_2}$, proving  part 
\eqref{iSfRgb}.

Next, we claim that
\begin{equation}\label{SlKyMk} \domin{\alpha_1}\cup\ldots\cup\domin{\alpha_n}=\set{x\in\At L:  x\leq 0_{\alpha_1\vee \cdots\vee \alpha_n} }\text.
\end{equation}
The ``$\subseteq$'' inclusion is an evident consequence of \eqref{blockscompute}. To prove the converse inclusion, assume that $x$ belongs to the righthand side of \eqref{SlKyMk}. Then, by \eqref{blockscompute} and distributivity,
\[
x=x\wedge 0_{\alpha_1\vee\cdots\vee \alpha_n}=
x\wedge (0_{\alpha_1}\vee\cdots\vee
0_{\alpha_n})=
(x\wedge 0_{\alpha_1})\vee\cdots\vee
(x\wedge 0_{\alpha_n})\text.
\]
By the join-irreducibility of $x$, there exists an $i\in\set{1,\ldots,n}$ such that $x= x\wedge 0_{\alpha_i}$. Hence $x\leq 0_{\alpha_i}$ implies that $x\in\domin{\alpha_i}$, proving \eqref{SlKyMk}. 

For $i\in\set{1,\ldots,n}$, $0_{\alpha_i}$ is the join of some (possibly only one) join-irreducible elements of $L$. These elements are necessarily atoms since $\length{\Ji L}=1$ and $J_{\alpha_i}\neq\emptyset$. 
Therefore all the $0_{\alpha_i}$ belong to $\zeta_0$, and so does their join, which is $0_{\alpha_1\vee\cdots\vee \alpha_n}$ by \eqref{blockscompute}. Since $\zeta_0$ is a boolean lattice, the number of atoms below $0_{\alpha_1\vee\cdots\vee \alpha_n}$, that is the size of the set given in \eqref{SlKyMk}, is $\length{[0,0_{\alpha_1\vee\cdots\vee \alpha_n}]}=\wxof0 {0_{\alpha_1\vee\cdots\vee \alpha_n}}$. This together with \eqref{SlKyMk} proves part \eqref{iSfRgc}.
\end{proof}

\subsection{A lemma on bipartite graphs}
In order to formulate a statement that we need in the proof of Theorem~\ref{thmmain}\eqref{thmmainb}\kern0.7pt-\eqref{thmmainc}, we have to associate a number-valued function with bipartite graphs.  By a finite \emph{directed bipartite  graph} we shall mean a structure $G=(U,X,E)$ where $U$ and $X$ are finite nonempty sets, referred to as upper and lower vertex sets, and $E\subseteq U\times X$ is an arbitrary relation. The power set, that is the set of all subsets, of $U$ is denoted by $P(U)$, and $P(X)$ has the analogous meaning. Let $\pinf U:=P(U)\cup\set{\infty}$ where $\infty$ is a symbol not in $P(U)$. For $V\in P(U)$, we let $\gd G(V):=\{x\in X:$ there is a $v\in V$ such that $(v,x)\in E\}$. This set is called the set of (lower) vertices \emph{dominated} by $V$. (We shall not use the word ``covered'' in this context since we want to avoid any confusion with the order-theoretic covering relation.) 
We define $\gd G(\infty):=X$. 
Let $\dstar GV$ stand for $|\gd G(V)|$; if $V=\set v$, then we write $\dstar Gv$ rather than $\dstar G{\set v}$. This way $\nopdstar G$, called the \emph{domination function} associated with $G$, is a
$\pinf U\to\mathbb N_0$ mapping. If $\phi\colon U\to U'$ is a bijection and $V\in\pinf U$, then $\phi(V):=\set{\phi(v): v\in V}$ for $V\in P(U)$ while $\phi(\infty):=\infty\in \pinf{U'}$.

\begin{lemma}\label{bpgraphl}
 Let $G=(U,X,E)$ and $G'=(U',X',E')$ be finite directed bipartite graphs. Then these two graphs are isomorphic if{f} there is a bijection $\phi\colon U\to U'$ that preserves the domination function, that is, $\dstar {G'}{\phi(V)}=\dstar GV$ holds for all $V\in\pinf U$.
\end{lemma}

\begin{proof} In order to prove the non-trivial direction of the lemma, assume that $\phi\colon U\to U'$ is a bijection that preserves the domination function. 
We associate two additional mappings with $G$ as follows:
{\allowdisplaybreaks
\begin{align*}
\gs G\colon P(U)\to P(X),\quad V\mapsto\set{x\in X: (&v,x)\in E \text{ for all }v\in V },\cr
\gee G\colon P(U)\to P(X),\quad V\mapsto\{x\in X: (&v,x)\in E \text{ for all }v\in V \text{ and}
\\*
(&u,x)\notin E \text{ for all }u\in U\setminus V\}\text.
\end{align*}%
}%
The corresponding number-valued functions are denoted by 
$\nopsstar G$ and $\nopestar G$, that is, $\sstar GV:=|\gs G(V)|$ and $\estar GV:=|\gee G(V)|$. These functions will be called the \emph{strong domination function} and the \emph{exact domination function}, respectively. 
Replacing $G$ by $G'$, we obtain the definition for $\nopsstar {G'}$ and $\nopestar {G'}$. Usually, we will elaborate our formulas only for $G$ since, sometimes implicitly, we will rely on the fact that the analogous formulas hold for $G'$ as well. 

Firstly, we prove that $\phi$ preserves the strong domination function. Let $V\in P(U)$; we show $\sstar {G'}{\phi(V)} = \sstar GV$ by induction on $|V|$. If $|V|=0$, then 
$\sstar G\emptyset=|X|=\dstar G{\infty}$, and the desired equality follows easily. The case $|V|=1$, say $V=\set{v}$, is even easier since $\sstar Gv=\dstar Gv$. Next, assume that $1<n\in\mathbb N$, $V=\set{v_1,\ldots,v_n}$ is an $n$-element subset of $U$, and the desired equality holds for all subsets of $U$ with less than $n$ elements. 
Based on the inclusion-exclusion principle, also called (logical) sieve formula, 
\[|T_1\cup\cdots\cup T_n| =   \sum_{i=1}^{n}(-1)^{i-1}
\sum_{\substack{I\subseteq \set{1,\ldots,n}\\|I|=i}}\,\Bigl|\bigcap_{j\in I} T_j)\Bigr|,
\]
and 
using that $\gs G$ agrees with $\gd G$ on singleton sets and satisfies the identity $\gs G(V_1)\cap \gs G(V_2)= \gs G(V_1\cup V_2)$, we can compute as follows.
{\allowdisplaybreaks
\begin{align*}
\dstar GV&=|\gd G(V)|=|\gd G(v_1)\cup\cdots\cup \gd G(v_n)|  
=|\gs G(v_1)\cup\cdots\cup \gs G(v_n)|  \cr
&=(-1)^{n-1} \Bigl|\bigcap_{j=1}^n \gs G(v_j)\Bigr|+ 
\sum_{i=1}^{n-1}(-1)^{i-1}
\sum_{\substack{I\subseteq \set{1,\ldots,n}\\|I|=i}}\Bigl|\bigcap_{j\in I}\gs G(v_j)\Bigr|\cr
&=(-1)^{n-1} \sstar GV+ 
\sum_{i=1}^{n-1}(-1)^{i-1}
\sum_{\substack{I\subseteq \set{1,\ldots,n}\\|I|=i}}
\sstar G{\set{v_j:j\in I }}\text.
\end{align*}
}%
Now,  $\phi$ preserves all summands but $(-1)^{n-1} \sstar GV$ in the previous line  by the induction hypothesis. Since  $\dstar GV$ is also preserved, we conclude that $\sstar GV$ is preserved either, completing the induction.
Thus, $\phi$ preserves the strong domination function.

Next, to show that $\phi$ preserves the exact domination function, let $V\in P(U)$. We want to show that $\estar{G'}{\phi(V)}=\estar GV$. This is clear if $V=U$ since 
$\estar GU=\sstar GU$ and the strong domination function is preserved. Hence we can assume that $V\neq U$. Let $W:=U\setminus V$. It is a $k$-element set for some $k\in\mathbb N$, so we can write $W=\set{w_1,\ldots,w_k}$. For $i=1,\ldots, k$, let $A_i:=\set{x\in \gs G(V): (w_i,x)\in E}$. Notice that $A_i=\gs G(V\cup \set{w_i})$.
Using the inclusion-exclusion principle again, we obtain that
{\allowdisplaybreaks
\begin{align*}
\estar GV &=|\gee G(V)|=\Bigl|\gs G(V)\setminus\bigcup_{i=1}^k A_i\Bigr| =
|\gs G(V)| - \Bigl|\bigcup_{i=1}^k A_i\Bigr|\cr
&=|\gs G(V)| + \sum_{i=1}^{k}(-1)^i
\sum_{\substack{I\subseteq \set{1,\ldots,k}\\|I|=i}}\Bigl| \bigcap_{j\in I}A_j\Bigr| \cr
&=\sstar GV + \sum_{i=1}^{k}(-1)^i
\sum_{\substack{I\subseteq \set{1,\ldots,k}\\|I|=i}} \Bigl|\gs G\bigl(V\cup\set{w_j:j\in I}\bigr)\Bigr|  \cr
&=\sstar GV + \sum_{i=1}^{k}(-1)^i
\sum_{\substack{I\subseteq \set{1,\ldots,k}\\|I|=i}} \sstar G { V\cup\set{w_j:j\in I}\bigr}\text. 
\end{align*}
}%
Therefore, since $\phi$ preserves the strong domination function, it preserves the exact domination function either. 

Now, we are ready to define an isomorphism $(\phi,\xi)\colon G\to G'$. That is, $\phi$ was originally given, and we intend to define a bijection $\xi\colon X\to X'$ such that $E'=\set{(\phi(u),\xi(x)): (u,x)\in E}$. Clearly, 
\begin{equation}\label{stpsisbEwd}
\gee G(V_1)\cap \gee G(V_2)=\emptyset \text{ for all } V_1\neq V_2\in P(U), \text{ and } X=\bigcup_{V\in P(U)}\gee G(V)\text;
\end{equation}
and the analogous assertion holds for $G'$.
Notice at this point that the elements of $X$ with degree 0 belong to $\gee G(\emptyset)$.  
For each $V\in P(U)$, let us fix a bijection $\xi_V\colon \gee G(V)\to\gee{G'}({\phi(V)})$; this is possible since $\estar GV=  \estar{G'}{\phi(V)}$. Then $\xi:=\bigcup_{V\in P(U)}\xi_V$ is an $X\to X'$ bijection by \eqref{stpsisbEwd}. 

Observe that the role of   $(G,\phi,\xi)$ and that of $(G',\phi^{-1},\xi^{-1})$ can be interchanged. Hence, in order to prove that $(\phi,\xi)$ is  an isomorphism, it suffices to show that $(\phi,\xi)$ sends edges to edges. To do so, assume that $(u,x)\in E$. Let $V:=\set{v\in U: (v,x)\in E}$. Then $u\in V$ and $x\in\gee G(V)$. Hence $\phi(u)\in\phi(V)$ and 
 $\xi(x)=\xi_V(x)\in  \gee{G'}({\phi(V)})$.
By the definition of $\gee{G'}$, this implies that $(\phi(u),\xi(x))$ belongs to $E'$, as desired.
\end{proof}

\subsection{The end of the proof}
Based on the auxiliary statements given so far, we are now in the position to complete the proof of the main result.
\begin{proof}[Proof of Theorem~\ref{thmmain} \eqref{thmmainb} and \eqref{thmmainc}]
Assume that $L$ and $L'$ are finite distributive lattices and 
\[(\psi,\kappa)\colon 
\bigl(\dskel L,\leq , \skel L,\eta_0,\eta_1,\wt\bigr) \to \bigl(\dskel {L'},\leq' , \skel {L'},\eta'_0,\eta'_1,\wt'\bigr)
\]
is an isomorphism between their weighted double skeletons.
Remember that  $\eta_0$ and $\eta_1$ were defined right after \eqref{dkselL},  the meaning of  $\eta_0'$ and $\eta_1'$ is analogous, and the diagram
\begin{equation}\label{cMutes}
\begin{aligned}
\includegraphics[scale=1.0]{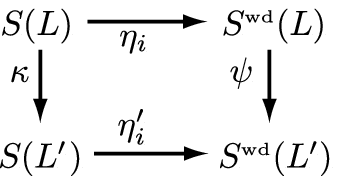}
\end{aligned}
\end{equation}
commutes for  $i\in\set{0,1}$.
Observe that
\begin{equation}\label{boThmxone}
\length{\Ji {L}}\leq 1 \quad\text{and}\quad \length{\Ji {L'}}\leq 
1\text.
\end{equation} 
Indeed, in case of  part \eqref{thmmainc} this is assumed.
In case of part \eqref{thmmainb}, the assumption together with 
the meaning of $\kappa$  implies that $L'$ is also $H^2$-irreducible, whence 
$\length{\Ji {L}}\leq 1$ and $\length{\Ji {L'}}\leq 1$ follow from  part \eqref{thmmaina}.

Firstly, assume that $|\skel L|=1$. Then $|\skel {L'}|= |\kappa(\skel L)| =1$.
It is well-known, and follows from Lemmas~\ref{joiNatmszn}  and \ref{kHhDwlP}\eqref{kHhDwlPc}, that  $L$ is boolean if{f} $1_L=\bigvee\At L$. Therefore, we obtain from Lemma~\ref{joiNatmszn} that $|\skel L|=1$ if{f} $L$ is boolean if{f} $\length{\Ji L}=0$, and the same holds for $L'$. Therefore 
\[\length L=\wxof {0_L}{1_L} = \wxofprime{\psi(0_L)}{\psi(1_L)}=  \wxofprime{0_{L'}}{1_{L'}} =  \length {L'},
\]
and we conclude that $L\cong L'$. Observe that the role of $L$ and that of $L'$ in the above argument can be interchanged, whence we also conclude that $\length{\Ji L}=0$ if{f} $\length{\Ji {L'}}=0$.

In the rest of the proof, we assume that $\length{\Ji L}=1$. Then, by the previous paragraph and \eqref{boThmxone}, $\length{\Ji {L'}}=1$ also holds.  We are going to define some auxiliary sets and structures associated with $\dskel L$; their ``primed'' counterparts associated with $\dskel {L'}$ are understood analogously.

Let $U:=\Jtop L$, $X:=\At L$, and $E:=\set{(u,x)\in U\times X: u>x}$. 
Notice that none of $U$, $X$ and $E$  is empty since $\length{\Ji L}=1$.   
Obviously, the directed bipartite graph $G:=(U,X,E)$ determines the poset $\Ji L$. Therefore, $G$ and $G'$  determine $L$ and $L'$, respectively, up to isomorphism. Consequently, by Lemma~\ref{bpgraphl}, it suffices to find a bijection
$\phi: U\to U'$ such that 
\begin{equation}\label{TFtarget}
\dstar {G'}{\phi(V)}=\dstar GV \text{ holds for all }  V\in \pinf  U\text.
\end{equation}
It follows from the commutativity of \eqref{cMutes} that, for $\gamma\in\skel L$,
\begin{equation}\label{psiNsupha}
\psi(0_\gamma)=\psi(\eta_0(\gamma))=
\eta_0'(\kappa(\gamma))=0_{\kappa(\gamma)}
\text{ and, similarly, }\psi(1_\gamma)=1_{\kappa(\gamma)}.
\end{equation}
For $\alpha\in \upskel L$, Lemma~\ref{nojiralpha} and \eqref{psiNsupha} imply  that  
{\allowdisplaybreaks
\begin{align*}
|J_\alpha| &= \sum_{\beta \leq \alpha, \,\, \beta \in \skel L} \mu_{\skel L}(\beta , \alpha) \cdot \wxof {0_\alpha} {1_\beta} \cr
&= \sum_{\beta \leq \alpha, \,\, \beta \in \skel L} \mu_{\skel {L'}}(\kappa(\beta) , \kappa(\alpha)) \cdot \wxofprime {\psi(0_\alpha)} {\psi(1_\beta)}  \cr
&= \sum_{\beta \leq \alpha, \,\, \beta \in \skel L} \mu_{\skel {L'}}(\kappa(\beta) , \kappa(\alpha)) \cdot \wxofprime {0_{\kappa(\alpha)}} {1_{\kappa(\beta)}}  \cr
&= \sum_{\beta' \leq \kappa(\alpha), \,\, \beta' \in \skel {L'}} \mu_{\skel {L'}}(\beta' , \kappa(\alpha)) \cdot \wxofprime {0_{\kappa(\alpha)}} {1_{\beta'}} =  |J_{\kappa(\alpha)}|\text. 
\end{align*}%
}%
This allows us to fix a bijection $\phi_\alpha: J_\alpha\to  J_{\kappa(\alpha)}$. (Notice that if $J_\alpha$ happens to be empty, then $\phi_\alpha=\emptyset$ is the empty mapping.)
Let $\phi$ be the union of all these $\phi_\alpha$, $\alpha \in\upskel L$. It follows from Lemma~\ref{kHhDwlP}\eqref{kHhDwlPb} that $\phi$ is a mapping.
Since the union of the corresponding $J_\alpha$, $\alpha\in\upskel L$, is $\Jtop L=U$ by Lemma~\ref{kHhDwlP}\eqref{kHhDwlPc},  and the analogous assertion holds for $U'$, $\phi\colon U\to U'$ is a bijective mapping.

It follows from Lemma~\ref{kHhDwlP}\eqref{kHhDwlPb}\kern0.7pt -\eqref{kHhDwlPc} that for each $u\in U$, there is a unique $\alpha(u)\in \upskel L$ such that $u\in J_{\alpha(u)}$. Similarly, for each $u'\in U'$, there is a unique $\alpha'(u')\in \upskel {L'}$ such that $u'\in J_{\alpha'(u')}$. The definition of $\phi$ implies that 
\begin{equation}\label{uOEWkH}
\alpha'\bigl(\phi(u)\bigr) = \kappa\bigl(\alpha(u)\bigr),\quad\text{for all}\quad u\in U\text.
\end{equation}

Assume that $u\in U$.  Then $u$ is not an atom of $L$, and its only lower cover in $L$ is $0_{\alpha(u)}$. Hence, for any $a\in \At L$, we have that $u>a \iff 0_{\alpha(u)}\geq a \iff a\in\domin {\alpha(u)}$. 
This yields that, for any $V\in P(U)$,
$\gd G(V) =\bigcup_{u\in V}\domin{\alpha(u)}
$. 
Therefore, taking the meaning of $\eta_0$ into account and using Lemma~\ref{iSfRg}, 
\begin{equation}\label{WGvztsZ}
\dstar GV=\Bigwxof {\eta_0(\zeta_0)}{\eta_0\bigl(\bigvee_{u\in V}\alpha(u) \bigr)} \text.
\end{equation}
Indicating the referenced formulas or their ``primed version'' at the equation signs and using that $( \psi,\kappa)$ preserves the extended weight function, 
we obtain that
{\allowdisplaybreaks
\begin{align*}
\dstar {G'}{\phi(V)}
&\ekval{(\ref{WGvztsZ})'}  \Bigwxofprime {\eta_0'(\zeta'_0)}{\eta_0'\bigl(\bigvee_{u\in V}\alpha'(\phi(u)) \bigr)} \cr
&\ekval{\eqref{uOEWkH}} \Bigwxofprime {\eta_0'(\kappa(\zeta_0))} {\eta_0'\bigl(\bigvee_{u\in V}\kappa(\alpha(u)) \bigr)}   \cr
&= \Bigwxofprime {\eta_0'(\kappa(\zeta_0))} {\eta_0'\Bigl(\kappa\bigl(\bigvee_{u\in V}\alpha(u)\bigr) \Bigr)} \cr
&\ekval{\eqref{cMutes}} \Bigwxofprime {\psi(\eta_0(\zeta_0))}{\psi\Bigl(\eta_0\bigl(\bigvee_{u\in V}\alpha(u)\bigr) \Bigr)} \cr
&=  \Bigwxof {\eta_0(\zeta_0)}{\eta_0\bigl(\bigvee_{u\in V}\alpha(u) \bigr)} \ekval{\eqref{WGvztsZ}} \dstar GV\text. 
\end{align*}
}%
This proves \eqref{TFtarget} for $V\in P(U)$.

We are left with the case $V=\infty\in\pinf U$. Then, using Lemma~\ref{kHhDwlP}\eqref{kHhDwlPc}, the validity of \eqref{TFtarget} is obtained as follows:
{\allowdisplaybreaks
\begin{align*}
\dstar G{\infty}&=|\At L|=|\At{\zeta_0}|= \wxof{0_{\zeta_0}} {1_{\zeta_0}} 
= \wxofprime{\psi(0_{\zeta_0})} {\psi(1_{\zeta_0})} \cr 
& \ekval{\eqref{psiNsupha}}
\wxofprime{0_{\kappa(\zeta_0)}} {1_{\kappa(\zeta_0)}}    
 = \wxofprime{0_{\zeta_0'}} {1_{\zeta_0'}}  \cr
&= |\At{\zeta_0'}|=|\At {L'}|=\dstar {G'}{\infty}= \dstar {G'}{\phi(\infty)}\text. \qedhere
\end{align*}% 
}%
\end{proof}

\end{document}